\documentclass[11pt]{amsart}
\usepackage{amssymb,amsmath}

\theoremstyle{plain}

\usepackage{latexsym}   
\usepackage{amssymb}    
\usepackage{amsmath}    
\usepackage{amsthm}     
\usepackage{mathrsfs}   
\newcommand{\beq}{\begin{equation}}
\newcommand{\eeq}{\end{equation}}
\newcommand{\bea}{\begin{eqnarray}}
\newcommand{\eea}{\end{eqnarray}}
\newcommand{\beas}{\begin{eqnarray*}}
\newcommand{\eeas}{\end{eqnarray*}}

\newtheorem{theorem}{Theorem}[section]
\newtheorem{definition}[theorem]{Definition}
\newtheorem{proposition}[theorem]{Proposition}
\newtheorem{corollary}[theorem]{Corollary}

\newtheorem{remark}[theorem]{Remark}
\newtheorem{example}[theorem]{Example}
\newtheorem{examples}[theorem]{Examples}
\newtheorem{foo}[theorem]{Remarks}

\newtheorem{assumption}[theorem]{Hypothesis}
\newcommand{\bG}{\mathbb G}

\newcommand{\p}{\partial}

\newcommand{\bM}{\mathbb M}

\newcommand{\Rn}{\mathbb R^n}

\newcommand{\di}{\mathfrak h}
\newcommand{\M}{\mathbb M}

\newcommand{\R}{\mathbb R}

\newcommand{\Ri}{\operatorname{Ric}}
\newcommand{\ve}{\varepsilon}

\title[Curvature-dimension inequalities, etc.]{Curvature-dimension inequalities and Li-Yau inequalities in sub-Riemannian spaces}

\author{Nicola Garofalo}
\address{Department of Mathematics\\
The Ohio State University\\ 
100 Math Tower\\ 231 West 18th Avenue\\ Columbus, OH 43210-1174} 
\email[Nicola
Garofalo]{rembrandt54@gmail.com}
\thanks{The work discussed in this paper was supported in part by NSF Grant DMS-1001317}

\dedicatory{Dedicated to Andrei Agrachev, on the occasion of his 60th birthday}

\begin{document}

\maketitle

\tableofcontents

\begin{abstract}
In this paper we present a survey of the joint program with Fabrice Baudoin originated with the paper \cite{BG1}, and continued with the works \cite{BG2}, \cite{BBG}, \cite{BG3} and \cite{BBGM}, joint with Baudoin, Michel Bonnefont and Isidro Munive.
\end{abstract}

\section{Introduction}\label{S:intro}
One of the most exciting aspects of Riemannian geometry consists in
the beautiful interplay between global topological and
geometric properties of the ambient manifold and properties
of solutions of those natural pde's such as the Laplace-Beltrami
operator $\Delta$, with its associated heat semigroup $P_t f(x) = e^{t\Delta} f(x)$.
In their 1986 Acta Mathematica paper \cite{LY} Li and Yau established their celebrated inequalities. Let us just focus on the one concerned with Ric $\ge 0$.

\begin{theorem}[The Li-Yau parabolic gradient estimate]\label{T:LYR} Suppose that $M$ is a complete, connected $n$-dimensional Riemannian manifold such that
$\Ri\ge 0$.
Then, for any $f\ge 0$ which solves the heat equation $\Delta f - f_t = 0$ on $M$ one has for $u = \ln f$,
\begin{equation}\label{LY}
|\nabla u|^2 - u_t \le \frac{n}{2t}.
\end{equation}
\end{theorem}
The motivation for \eqref{LY} comes from considering the case when $M$ is flat $\Rn$ and $f(x,t) = (4\pi t)^{-\frac{n}{2}} \exp(-\frac{|x|^2}{4t})$ is the fundamental solution of the heat equation. In such case $u = \ln f$ is easily seen to satisfy 
\[
|\nabla u|^2 - u_t \equiv \frac{n}{2t}.
\]
Understanding the ``$\le$'' in \eqref{LY} requires a deeper analysis of the role played by curvature. 
Integration of the  the Li-Yau inequality \eqref{LY} along a geodesic path joining $(y,t)$ to $(x,s)$, where $x, y\in M$ and $0<s<t$, gives the following fundamental result.  
\begin{theorem}[The Li-Yau Harnack inequality]\label{T:harnackR}
Let $M$ be a complete connected $n$-dimensional Riemannian manifold having \emph{Ric}$\ \ge 0$. Let $f\ge 0$ be a solution of the heat equation on $M$.
For any $x, y \in M$, $0<s<t<\infty$, one has
\[
f(x,s) \le f(y,t) \left(\frac{t}{s}\right)^{\frac{n}{2}} \exp\left(\frac{d(x,y)^2}{4(t-s)}\right).
\]
\end{theorem}
Theorem \ref{T:harnackR} extends to Riemannian manifolds  with $\Ri\ge 0$ the Harnack inequality for the heat equation independently discovered by B. Pini in \cite{P} and J. Hadamard in \cite{Ha}.
Theorems \ref{T:LYR}  and \ref{T:harnackR} provide remarkable evidence of how the geometry of
the manifold is intimately connected to the properties of its
Laplacian and the associated heat flow. In fact, once Theorem \ref{T:harnackR} is available one can obtain many fundamental results, such as Liouville type theorems, on and off-diagonal Gaussian upper bounds for the heat kernel, Sobolev and isoperimetric inequalities, etc.

Another beautiful global result which connects the geometry to the topology of $M$ is the Bonnet-Myers theorem which states that if for some $\rho_1>0$,  Ric$ \ge (n-1) \rho_1$, then $M$ with its Riemannian metric is compact, with a finite
fundamental group, and diam$(M) \le \frac{\pi}{\sqrt{\rho_1}}$.

\subsubsection{The identity of Bochner and the role of Jacobi fields}

The original proof of Li and Yau of Theorem \ref{T:LYR} hinges on two basic tools from Riemannian geometry:
\begin{itemize}
\item[(i)] the Bochner identity
\begin{equation}\label{b}
\Delta(|\nabla f|^2) = 2 \|\nabla^2 f\|^2 + 2 <\nabla f,\nabla(\Delta f)> + 2 \Ri(\nabla f,\nabla f),
\end{equation}
which holds for any $f\in C^3(M)$; 

\vskip 0.2in 

\item[(ii)] the Laplacian comparison theorem. When  $\Ri \ge 0$ the latter states that  the geodesic distance on $M$ satisfies the following differential inequality outside the cut-locus of a fixed base point (and in the sense of distributions on $M$)
\begin{equation}\label{lapR}
\Delta \rho_M(x) \le  \frac{n-1}{\rho_M(x)}.
\end{equation}
\end{itemize} 
As it is well-known, the Laplacian comparison theorem (like other  comparison theorems in Riemannian geometry, or like the Bonnet-Myers theorem) uses in an essential way the existence of a rich supply of Jacobi fields. 

This paper is devoted to surveying  a joint program with Fabrice Baudoin originated with the paper \cite{BG1}, and continued with the works \cite{BG2}, \cite{BBG}, \cite{BG3} and \cite{BBGM}. It is worth emphasizing that our approach allows for the first time to extend the Li-Yau program, and many of its fundamental consequences, to situations which are genuinely non-Riemannian. 
The original motivation in \cite{BG1} was generalizing global results such as Theorems  \ref{T:LYR}  and \ref{T:harnackR} above, or the topological Bonnet-Myers theorem, to smooth manifolds in which the governing operator is no longer the Laplace-Beltrami operator, but rather a smooth locally subelliptic operator $L$. These operators are typically never elliptic and their natural geometric framework is that of sub-Riemannian manifolds. Such manifolds are a generalization of Riemannian ones and they constitute the appropriate setting for describing phenomena with a constrained dynamic, in which only certain directions in the tangent space are allowed.

We close this introduction by mentioning that, in their interesting preprint \cite{AL},  Agrachev and Lee have used a notion of Ricci tensor, denoted by $\mathfrak{Ric}$, which was introduced by the first author in \cite{A}. They study three-dimensional contact manifolds and, under the assumption that the manifold be Sasakian, they prove that a lower bound on $\mathfrak{Ric}$ implies the so-called measure-contraction property. In particular, when $\mathfrak{Ric} \ge 0$, then the manifold $\M$ satisfies a global volume growth similar to the Riemannian Bishop-Gromov theorem. An analysis shows that, interestingly, our notion of Ricci tensor coincides, up to a scaling factor, with theirs. 

We also mention that for three-dimensional contact manifolds, the sub-Riemannian geometric  invariants were computed by Hughen in his unpublished Ph.D. dissertation, see \cite{Hu}. In particular, with his notations, the  CR Sasakian structure corresponds to the case $a_1^2+a_2^2=0$ and, up to a scaling factor, his $K$ is the Tanaka-Webster Ricci curvature. In such respect, the Bonnet-Myers type theorem obtained  by Hughen (Proposition 3.5 in \cite{Hu})  is the exact analogue (with a better constant) of our Theorem \ref{T:BM}, applied to the case of three-dimensional  Sasakian manifolds. Finally, it must be mentioned that a Bonnet-Myers type theorem on general three-dimensional CR manifolds was first obtained by Rumin in \cite{rumin}. The methods of Rumin and Hughen are close as they both rely on the analysis of the second-variation formula for sub-Riemannian geodesics.

\medskip

\noindent \textbf{Acknowledgment:} I would like to thank Ugo Boscain, Mario Sigalotti, Andrey Sarychev, Davide Barilari and Dario Prandi for their gracious invitation to speak at the \emph{INDAM meeting on 
Geometric Control and 
sub-Riemannian Geometry}, held at the Palazzone in Cortona, May 21-25, 2012. I am very grateful for having been offered this opportunity to honor Andrei Agrachev.

\section{From Riemannian to sub-Riemannian geometry}

A fundamental property of the Laplace-Beltrami operator is \emph{ellipticity}. As we have just said, in sub-Riemannian geometry the relevant partial differential operators, the sub-Laplacians, fail to be elliptic. The moment one gives up coercivity (i.e., control of all directions in the tangent space), new interesting phenomena arise. For instance, the exponential mapping fails to be a local diffeomorphism, and geodesics are no longer locally unique. A rich theory of Jacobi fields is (at least presently) not available and, consequently, results such as the Laplacian comparison theorem, the Bonnet-Myers theorem, or Theorems \ref{T:LYR}  and \ref{T:harnackR} seemed to be completely out of reach. Furthermore, it was not clear what one means by ``Ricci curvature''.

The paper \cite{BG1} took a different approach to these questions, based on a new \emph{curvature-dimension inequality} and a systematic use of the heat semigroup. Besides the Riemannian case, the program in \cite{BG1} presently covers sub-Riemannian spaces of rank two, such as, for instance, Carnot groups of step two, CR manifolds, etc. This is  the first genuinely non-Riemannian setting in which a good notion of \emph{Ricci curvature} has been introduced, and we feel it is important to be emphasize that the Riemannian approach has been so far mostly unsuccessful to cover the large classes of examples encompassed by \cite{BG1}.

In this connection we stress that, even in the Riemannian framework, the ideas in \cite{BG1} provide a new and simplified account of the Li-Yau program based on tools which are purely analytical and avoid the use of results which are preeminently based on the theory of Jacobi fields, such as, e.g., the Laplacian or the volume comparison theorem, see \cite{BG2}.

\section{The curvature-dimension inequality CD$(\rho,n)$ and the Ricci tensor}\label{S:cdr}

Recall that a Riemannian manifold $M$ with Laplacian $\Delta$ is said to satisfy the Bakry-Emery \emph{curvature-dimension inequality} CD$(\rho,n)$ if
\begin{equation}\label{cd}
\Gamma_2(f) \ge \frac 1n (\Delta f)^2 + \rho \Gamma(f),\ \ \ \ \forall f\in C^\infty(M).
\end{equation}
Here, 
\begin{equation}\label{gammas}
\Gamma(f) = \frac 12\left\{\Delta(f^2) - 2 f \Delta f\right\} = |\nabla f|^2,\ \ \ \Gamma_2(f) = \frac 12\left\{\Delta(\Gamma(f)) - 2 \Gamma(f, \Delta f)\right\}.
\end{equation}
Using Bochner's identity \eqref{b} and Newton's inequality, it is easy to see that if Ric$\ \ge \rho$, then CD$(\rho,n)$ holds. 
It is remarkable that the curvature dimension inequality \eqref{cd} perfectly  captures the
notion of Ricci lower bound. It was in fact proved by Bakry in Proposition 6.2 in \cite{B} that: \emph{on a $n$-dimensional Riemannian manifold $\bM$ the inequality \emph{CD}$(\rho,n)$ implies} Ric $\ge \rho$. In conclusion, 
\begin{equation}\label{cdR}
\text{Ric}\ge \rho \Longleftrightarrow \text{CD}(\rho,n).
\end{equation}

This equivalence \eqref{cdR} was the motivation behind the work \cite{BG1}, whose setup we now describe.

We consider a smooth, connected manifold $\bM$ endowed with a smooth measure $\mu$ and a smooth second-order diffusion operator $L$ with real coefficients, satisfying $L1=0$, and which is symmetric with respect to $\mu$ and non-positive. By this we mean that 
\begin{equation}\label{sa}
\int_\bM f L g d\mu=\int_\bM g Lf d\mu,\ \ \ \ \ \ \int_\bM f L f d\mu \le 0,
\end{equation}
for every $f , g \in C^ \infty_0(\bM)$. We make the technical assumption that $L$ be locally subelliptic in the sense of \cite{FP}, and associate with $L$ the following symmetric, first-order, differential bilinear form: 
\begin{equation}\label{gamma}
\Gamma(f,g) =\frac{1}{2}\left\{L(fg)-fLg-gLf\right\}, \quad f,g \in C^\infty(\bM).
\end{equation}
The expression $\Gamma(f) = \Gamma(f,f)$ is known as \textit{le carr\'e du champ}, see \eqref{gammas}
 above. There is a canonical distance associated with the operator $L$:
\begin{equation}\label{di}
d(x,y)=\sup \left\{ |f(x) -f(y) | \mid f \in  C^\infty(\bM) , \| \Gamma(f) \|_\infty \le 1 \right\},\ \ \  \ x,y \in \bM,
\end{equation}
where for a function $g$ on $\bM$ we have let $||g||_\infty = \underset{\bM}{\text{ess} \sup} |g|$.  A tangent vector $v\in T_x\M$ is called \emph{subunit} for $L$ at $x$ if   
$v = \sum_{i=1}^m a_i X_i(x)$, with $\sum_{i=1}^m a_i^2 \le 1$, see \cite{FP}. A Lipschitz path $\gamma:[0,T]\to \M$ is called subunit for $L$ if $\gamma'(t)$ is subunit for $L$ at $\gamma(t)$ for a.e. $t\in [0,T]$. We then define the subunit length of $\gamma$ as $\ell_s(\gamma) = T$. Given $x, y\in \M$, we indicate with 
\[
S(x,y) =\{\gamma:[0,T]\to \M\mid \gamma\ \text{is subunit for}\ L, \gamma(0) = x,\ \gamma(T) = y\}.
\]
In this paper we make the assumption that 
\begin{equation}\label{ls}
S(x,y) \not= \varnothing,\ \ \ \ \text{for every}\ x, y\in \M.
\end{equation}
Under such hypothesis one verifies that
\begin{equation}\label{ds}
d_s(x,y) = \inf\{\ell_s(\gamma)\mid \gamma\in S(x,y)\},
\end{equation}
defines a true distance on $\M$, and that furthermore, 
\[
d(x,y) = d_s(x,y),\ \ \ x, y\in \mathbb M.
\]
It follows that one can work indifferently with either one of the distances $d$ in \eqref{di}, 
or $d_s$ in \eqref{ds}. 

Throughout this paper we assume that the metric space $(\M,d)$ be complete.

We also suppose that $\bM$ is equipped with a symmetric, first-order differential bilinear form $\Gamma^Z:C^\infty(\bM)\times C^\infty(\bM) \to \R$, satisfying
\[
\Gamma^Z(fg,h) = f\Gamma^Z(g,h) + g \Gamma^Z(f,h).
\]
We assume that $\Gamma^Z(f) = \Gamma^Z(f,f) \ge 0$ (one should notice that $\Gamma^Z(1) = 0$). 

Given the sub-Laplacian $L$ and the first-order bilinear form $\Gamma^Z$ on $\bM$, we now introduce the following second-order differential forms:
\begin{equation}\label{gamma2}
\Gamma_{2}(f,g) = \frac{1}{2}\big\{L\Gamma(f,g) - \Gamma(f,
Lg)-\Gamma (g,Lf)\big\},
\end{equation}
\begin{equation}\label{gamma2Z}
\Gamma^Z_{2}(f,g) = \frac{1}{2}\big\{L\Gamma^Z (f,g) - \Gamma^Z(f,
Lg)-\Gamma^Z (g,Lf)\big\}.
\end{equation}
Observe that if $\Gamma^Z\equiv 0$, then $\Gamma^Z_2 \equiv 0$ as well. As for $\Gamma$ and $\Gamma^Z$, we will use the notations  $\Gamma_2(f) = \Gamma_2(f,f)$, $\Gamma_2^Z(f) = \Gamma^Z_2(f,f)$.


We are ready to introduce the central character of our program, a generalization of the above mentioned curvature-dimension inequality \eqref{cdR}.

\begin{definition}\label{D:cdi}
We say that $\M$ satisfies the \emph{generalized curvature-dimension inequality} \emph{CD}$(\rho_1,\rho_2,\kappa,d)$ with respect to $L$ and $\Gamma^Z$ if there exist constants $\rho_1 \in \mathbb{R}$, $\rho_2 >0$, $\kappa \ge 0$, and $0<d\le \infty$ such that the inequality 
\begin{equation}\label{cdi}
\Gamma_2(f) +\nu \Gamma_2^Z(f) \ge \frac{1}{d} (Lf)^2 +\left( \rho_1 -\frac{\kappa}{\nu} \right) \Gamma(f) +\rho_2 \Gamma^Z(f)
\end{equation}
 hold for every  $f\in C^\infty(\bM)$ and every $\nu>0$.
\end{definition}

It is worth observing explicitly that if in Definition \ref{D:cdi} we choose $L = \Delta$, $\Gamma^Z \equiv 0$,  $d = n=$ dim($\bM$), $\rho_1 = \rho$ and $\kappa = 0$, we obtain the Riemannian curvature-dimension inequality CD$(\rho,n)$ in \eqref{cdR} above. Thus, the case of Riemannian manifolds is trivially encompassed by Definition \ref{D:cdi}. 
We also remark that, changing $\Gamma^Z$ into $a \Gamma^Z$, where $a>0$, changes the inequality  CD$(\rho_1,\rho_2,\kappa,d)$ into CD$(\rho_1,a \rho_2, a \kappa,d)$. We express this fact by saying that the quantity $\frac{\kappa}{\rho_2}$ is intrinsic. 
Hereafter, when we say that  $\M$ satisfies the curvature dimension inequality CD$(\rho_1,\rho_2,\kappa,d)$ (with respect to $L$ and $\Gamma^Z$), we will routinely avoid repeating at each occurrence the sentence ``for some $\rho_2>0$, $\kappa\ge 0$ and $d >0$''. Instead, we will explicitly mention whether $\rho_1 = 0$, or $>0$, or simply $\rho_1\in \R$. 
The reason for this is that the parameter $\rho_1$ in the inequality \eqref{cdi} has a special relevance since, in the geometric examples in \cite{BG1}, it represents the lower bound on a sub-Riemannian generalization of the Ricci tensor. Thus, $\rho_1 = 0$ is, in our framework, the counterpart of  the Riemannian Ric\ $\ge 0$, whereas when $\rho_1 >0$ $(<0)$, we are dealing with the counterpart of the case Ric\ $>0$ (Ric  bounded from below by a negative constant).

In addition to \eqref{cdi} we will work with three general assumptions: they will be listed as Hypothesis \ref{A:exhaustion}, \ref{A:main_assumption} and \ref{A:regularity}. 

\begin{assumption}\label{A:exhaustion}
There exists an increasing
sequence $h_k\in C^\infty_0(\bM)$   such that $h_k\nearrow 1$ on
$\bM$, and \[
||\Gamma (h_k)||_{\infty} +||\Gamma^Z (h_k)||_{\infty}  \to 0,\ \ \text{as} \ k\to \infty.
\]
\end{assumption}
We will also assume that the following commutation relation be satisfied.

\begin{assumption}\label{A:main_assumption} 
For any $f \in C^\infty(\bM)$ one has
\[
\Gamma(f, \Gamma^Z(f))=\Gamma^Z( f, \Gamma(f)).
\]
\end{assumption}

When $\bM$ is a Riemannian manifold, $\mu$ is the Riemannian volume on $\bM$, and $L = \Delta$, then $d(x,y)$ in \eqref{di} above is equal to the Riemannian distance on $\bM$. In this situation if we take $\Gamma^Z\equiv 0$, then Hypothesis \ref{A:exhaustion}, \ref{A:main_assumption} are fulfilled. In fact, Hypothesis \ref{A:main_assumption} is trivially satisfied, whereas Hypothesis \ref{A:exhaustion} is equivalent to assuming that $(\bM,d)$ be a complete metric space, which we are assuming. 

Before proceeding with the discussion, we pause to stress that, in the generality in which we work the bilinear differential form $\Gamma^Z$,  unlike $\Gamma$, is not a priori canonical. Whereas $\Gamma$ is determined once $L$ is assigned, the form $\Gamma^Z$ in general is not intrinsically associated with $L$. However, in the geometric examples described in section 2 of the paper \cite{BG1} the choice of $\Gamma^Z$ is canonical, as is the case, for instance, for CR Sasakian manifolds. The reader should think of $\Gamma^Z$ as an orthogonal complement of $\Gamma$: the bilinear form $\Gamma$ represents the square of the length of the gradient in the horizontal directions, whereas $\Gamma^Z$ represents the square of the length of the gradient along the vertical directions. 


We will also need the following assumption which is necessary to rigorously justify the computations in \cite{BG1} on functionals of the heat semigroup. Hereafter, we will denote  by $P_t = e^{tL}$ the semigroup generated by the diffusion operator $L$. 
\begin{assumption}\label{A:regularity}
The semigroup $P_t$ is stochastically complete that is, for $t \ge 0$, $P_t 1=1$ and  for every $f \in C_0^\infty(\bM)$ and $T \ge 0$, one has 
\[
\sup_{t \in [0,T]} \| \Gamma(P_t f)  \|_{ \infty}+\| \Gamma^Z(P_t f) \|_{ \infty} < +\infty.
\]

\end{assumption}

In the Riemannian setting ($L = \Delta$ and $\Gamma^Z \equiv 0$), Hypothesis \ref{A:regularity}, is satisfied if one assumes the lower bound Ricci$\ \ge \rho$, for some $\rho \in \R$. This can be derived from the paper by Yau \cite{Yau2} and Bakry's note \cite{bakry-CRAS}. It thus follows that, in the Riemannian case, the Hypothesis \ref{A:regularity} is not needed since it can be derived as a consequence of the curvature-dimension inequality CD$(\rho,n)$ in \eqref{cdR} above. More generally, it is proved in \cite{BG1} that a similar situation occurs in every sub-Riemannian manifold with transverse symmetries  of Yang-Mills type, for the relevant definitions see \cite{BG1}. In that paper it is shown that, in such framework, the Hypothesis \ref{A:regularity} is not needed since it follows (in a non-trivial way) from the generalized curvature-dimension inequality CD$(\rho_1,\rho_2,\kappa,d)$ in Definition \ref{D:cdi} above.  


The above discussion prompts us to underline the distinctive aspect of the theory developed in the papers \cite{BG1},  \cite{BG3}, \cite{BBG} and \cite{BBGM}: \emph{for the class of complete sub-Riemannian manifolds with transverse symmetries of Yang-Mills type studied in \cite{BG1}, all the results are solely deduced from the curvature-dimension inequality \emph{CD}$(\rho_1,\rho_2,\kappa,d)$ in \eqref{cdi}}.

\section{Li-Yau type estimates}

In this section, we discuss a generalization of the celebrated Li-Yau inequality in
\cite{LY} to the heat semigroup associated with the subelliptic
operator $L$. We mention that, in this setting, related
inequalities were obtained by Cao-Yau \cite{Cao-Yau}. However, these
authors work locally and the geometry of the manifold does not enter in their study.
Instead, the analysis in \cite{BG1} in based on some entropic inequalities which are derived from the curvature-dimension inequality \eqref{cdi} above. We have mentioned in the introduction that, even when specialized to the Riemannian case, the ideas in this section provide a new, more elementary approach to the Li-Yau inequalities. For this aspect we refer the reader to the paper \cite{BG3}.

\begin{theorem}[sub-Riemannian Li-Yau gradient estimate]\label{T:ge}
Assume that the curvature-dimension inequality \eqref{cdi} be satisfied for $\rho_1\in \R$, and that the  Hypothesis  \ref{A:exhaustion},  \ref{A:main_assumption}, \ref{A:regularity}  hold. Let  $f \in C_0^\infty(\M)$, $f  \ge 0$, $f \not\equiv 0$, then the following inequality holds for $t>0$:
\[
\Gamma (\ln P_t f) +\frac{2 \rho_2}{3}  t \Gamma^Z (\ln P_t f) \le
\left(1+\frac{3\kappa}{2\rho_2}-\frac{2\rho_1}{3} t\right)
\frac{LP_t f}{P_t f} +\frac{d\rho_1^2}{6} t-\frac{\rho_1 d}{2}\left(
1+\frac{3\kappa}{2\rho_2}\right) +\frac{d\left(
1+\frac{3\kappa}{2\rho_2}\right)^2}{2t}.
\]
\end{theorem}

\begin{remark}\label{R:rho1zero}
We notice that when $\rho_1\ge \rho_1'$, then one trivially has that:
\[
\emph{CD}(\rho_1,\rho_2,\kappa,d)\  \Longrightarrow\  \emph{CD}(\rho_1',\rho_2,\kappa,d).
\]
As a consequence of this observation, when \eqref{cdi} holds with $\rho_1>0$, then also \text{CD}$(0,\rho_2,\kappa,d)$ is true. Therefore, when $\rho_1\ge 0$, Theorem \ref{T:ge} gives in particular for $f\in C_0^\infty(\M)$, $f\ge 0$,
\begin{align}\label{liyaupositifzero}
\Gamma (\ln P_t f) +\frac{2 \rho_2}{3}  t \Gamma^Z (\ln P_t f) \le
\left(1+\frac{3\kappa}{2\rho_2}\right) \frac{LP_t f}{P_t f} +\frac{d\left(
1+\frac{3\kappa}{2\rho_2}\right)^2}{2t}.
\end{align}
However, this inequality is not optimal when $\rho_1>0$. It leads to a optimal Harnack inequality only when
$\rho_1 = 0$. 
\end{remark}

\begin{remark}\label{R:ultrac}
Throughout the remainder of the paper the symbol $D$ will only be used with the following meaning:
 \begin{equation}\label{D} D =  d  \left( 1+
\frac{3\kappa}{2\rho_2}\right). \end{equation}
With this notation, observing that the left-hand side of \eqref{liyaupositifzero} is always nonnegative, and that $LP_t f = \p_t P_t f$, when $\rho_1 \ge 0$ we obtain
\begin{align}\label{ultracontractivity}
\p_t(\ln (t^{D/2} P_t f(x))) \ge 0.
\end{align}
By integrating \eqref{ultracontractivity} from $t<1$ to $1$ leads to the following on-diagonal bound for the heat kernel,
\begin{align}\label{ultracontractivity2}
p (x,x,t) \le \frac{1}{t^{D/2}} p(x,x,1).
\end{align}

The constant $\frac{D}{2}$ in \eqref{ultracontractivity2} is not optimal, in
general, as the example of the heat semigroup on a Carnot group shows. In such case, in fact, one can show that the heat kernel $p(x,y,t)$ is homogeneous
of degree $-\frac{Q}{2}$ with respect to the non-isotropic group dilations,
where $Q$ indicates the corresponding
homogeneous dimension of the group. From such homogeneity of $p(x,y,t)$,
one obtains the estimate
\begin{align*}
p (x,x,t) \le \frac{1}{t^{Q/2}} p(x,x,1),
\end{align*}
which, unlike \eqref{ultracontractivity2}, is best possible.
In the sub-Riemannian setting it does not seem easy to obtain sharp geometric constants by using only
the curvature-dimension inequality \eqref{cdi}. This aspect is quite different
from the Riemannian case.
\end{remark}

\section{The parabolic Harnack inequality for Ricci $\ge 0$}\label{S:harnack}

In this section we discuss a generalization of the celebrated Harnack inequality in
\cite{LY} to solutions of the heat equation $Lu - u_t = 0$
on $\bM$. One
should also see the paper \cite{Cao-Yau}, where the authors deal
with subelliptic operators on a compact manifold. As we have
mentioned, these authors do not obtain bounds which depend 
on the sub-Riemannian geometry of the underlying manifold. Henceforth, we indicate with $C_b^\infty(\M)$ the space $C^\infty(\M)\cap L^\infty(\M)$.

\begin{theorem}\label{T:harnack}
Assume that the curvature-dimension inequality \eqref{cdi} be satisfied for $\rho_1\ge 0$ and that the Hypothesis \ref{A:exhaustion}, \ref{A:main_assumption}, \ref{A:regularity} hold. Given $(x,s), (y,t)\in \bM\times (0,\infty)$, with
$s<t$, one has for any $f\in
C_b^\infty(\bM)$, $f\ge 0$,
\begin{equation}\label{beauty}
P_s f(x) \le P_t f(y) \left(\frac{t}{s}\right)^{\frac{D}{2}} \exp\left(
\frac{D}{d} \frac{d(x,y)^2}{4(t-s)} \right).
\end{equation}
\end{theorem}

\begin{proof}
Let $f\in C_0^\infty(\M)$ be as in the statement of the theorem, and for every $(x,t)\in \bM\times
(0,\infty)$ consider $u(x,t) = P_t f(x)$ . Since $Lu = \frac{\p u}{\p t}$, in terms of $u$
the inequality \eqref{liyaupositifzero} can be reformulated as
\[ \Gamma (\ln u) +\frac{2 \rho_2}{3} t \Gamma^Z (\ln u) \le
(1+\frac{3\kappa}{2\rho_2}) \frac{\p \log u}{\p t}  +
\frac{d\left( 1+\frac{3\kappa}{2\rho_2}\right)^2}{2t}.
\]
Recalling \eqref{D}, this implies in particular, 
\begin{equation}\label{liyaupositif2}
- \frac{\p \ln u}{\p t} \le - \frac{d}{D} \Gamma
(\ln u)  +\frac{D}{2t}.
\end{equation}

We now fix two points $(x,s), (y,t)\in \bM\times (0,\infty)$, with
$s<t$. Let $\gamma(\tau)$, $0\le \tau \le T$ be a subunit path such
that $\gamma(0) = y$, $\gamma(T) = x$, and consider the path in $\bM\times
(0,\infty)$ defined by
\[
\alpha(\tau) = \left(\gamma(\tau),t + \frac{s-t}{T}\tau\right),\ \ \
\ 0\le \tau\le T,
\]
so that $\alpha(0) = (y,t)$, $\alpha(T) = (x,s)$. We have
\begin{align*}
\ln \frac{u(x,s)}{u(y,t)}& = \int_0^T \frac{d}{d\tau} \ln
u(\alpha(\tau)) d\tau \\
& \le \int_0^T \left[\Gamma(\ln u(\alpha(\tau)))^{\frac{1}{2}}
- \frac{t-s}{T} \frac{\p \ln u}{\p t}(\alpha(\tau))\right] d \tau.
\end{align*}
Applying \eqref{liyaupositif2} for any $\epsilon >0$ we find
\begin{align*}
\log \frac{u(x,s)}{u(y,t)}& \le T^{\frac{1}{2}} \left(\int_0^T
\Gamma(\ln u)(\alpha(\tau)) d\tau\right)^{\frac{1}{2}} -
\frac{t-s}{T} \int_0^T \frac{\p \ln u}{\p t}(\alpha(\tau)) d \tau
\\
& \le \frac{1}{2\epsilon} T + \frac{\epsilon}{2} \int_0^T \Gamma(\ln
u)(\alpha(\tau)) d\tau - \frac dD \frac{t-s}{T}
\int_0^T \Gamma(\ln u)(\alpha(\tau)) d\tau
\\
&   - \frac{D(s-t)}{2T}
\int_0^T \frac{d\tau}{t + \frac{s-t}{T} \tau}.
\end{align*}
If we now choose $\epsilon >0$ such that
\[
\frac{\epsilon}{2} = \frac dD\frac{t-s}{T},
\]
we obtain from the latter inequality
\[
\log \frac{u(x,s)}{u(y,t)} \le
\frac Dd\frac{\ell_s(\gamma)^2}{4(t-s)}  +
\frac{D}{2}
\ln\left(\frac{t}{s}\right),
\]
where we have denoted by $\ell_s(\gamma)$ the subunitary length of
$\gamma$. If we now minimize over all subunitary paths joining $y$
to $x$, and we exponentiate, we obtain
\[
u(x,s) \le u(y,t) \left(\frac{t}{s}\right)^{\frac{D}{2}} \exp\left(\frac Dd \frac{d(x,y)^2}{4(t-s)} \right).
\]
This proves
\eqref{beauty} when $f \in C_0^\infty(\M)$. We can then extend the result to  $f \in C_b^\infty(\bM)$ by considering the approximations $h_n P_\tau f \in C_0^\infty(\M)$ , where $h_n \in C_0^\infty(\M)$, $h_n \ge 0$, $h_n \underset{n\to \infty}{\to}\ 1$, and let $n \to \infty$ and $\tau \to 0$.

\end{proof}

The following result represents an important consequence of Theorem
\ref{T:harnack}.

\begin{corollary}\label{C:harnackheat}
Suppose that the curvature-dimension inequality \eqref{cdi} be satisfied for $\rho_1\ge 0$, and that the Hypothesis \ref{A:exhaustion}, \ref{A:main_assumption},  \ref{A:regularity} be valid. Let $p(x,y,t)$ be the heat kernel on $\bM$. For every $x,y, z\in
\bM$ and every $0<s<t<\infty$ one has
\[
p(x,y,s) \le p(x,z,t) \left(\frac{t}{s}\right)^{\frac{D}{2}}
\exp\left(\frac{D}{d} \frac{d(y,z)^2}{4(t-s)} \right).
\]
\end{corollary}

\section{Off-diagonal Gaussian upper bounds for Ricci $\ge 0$}\label{S:gaussianub}

Suppose that the assumption of Theorem \ref{T:harnack} are in force. Fix $x\in \bM$ and
$t>0$. Applying Corollary \ref{C:harnackheat} to $(y,t)\to p(x,y,t)$
for every $y\in B(x,\sqrt t)$ we find
\[
p(x,x,t) \le  2^{\frac{D}{2}} e^{\frac{D}{4d}}\ p(x,y,2t) =
C(\rho_2,\kappa,d) p(x,y,2t).
\]
Integration over $B(x,\sqrt t)$ gives
\[
p(x,x,t)\mu(B(x,\sqrt t)) \le C(\rho_2,\kappa,d) \int_{B(x,\sqrt
t)}p(x,y,2t)d\mu(y) \le C(\rho_2,\kappa,d),
\]
where we have used $P_t1\le 1$. This gives the on-diagonal upper
bound 
\begin{equation}\label{odub} 
p(x,x,t) \le \frac{C(\rho_2,\kappa,d)}{\mu(B(x,\sqrt t))}.
\end{equation}

Obtaining an off-diagonal upper bound for the heat kernel requires a more delicate analysis. The relevant result is contained in the following theorem, for whose proof we refer the reader to \cite{BG1}. 

\begin{theorem}\label{T:ub}
Assume that the curvature-dimension inequality \eqref{cdi} be satisfied for $\rho_1\ge 0$ and that the Hypothesis \ref{A:exhaustion}, \ref{A:main_assumption},  \ref{A:regularity} be fulfilled. For any $0<\epsilon <1$
there exists a constant $C(\rho_2,\kappa,d,\epsilon)>0$, which tends
to $\infty$ as $\epsilon \to 0^+$, such that for every $x,y\in \bM$
and $t>0$ one has
\[
p(x,y,t)\le \frac{C(d,\kappa,\rho_2,\epsilon)}{\mu(B(x,\sqrt
t))^{\frac{1}{2}}\mu(B(y,\sqrt t))^{\frac{1}{2}}} \exp
\left(-\frac{d(x,y)^2}{(4+\epsilon)t}\right).
\]
\end{theorem}

\section{A sub-Riemannian Bonnet-Myers theorem}\label{S:myer}

Let  $(\mathbb{M},g)$ be a complete, connected Riemannian manifold
of dimension $n\ge 2$. It is well-known that if for some $\rho>0$ the Ricci tensor of
$\mathbb{M}$ satisfies the bound 
\begin{equation}\label{ricbd}
\text{Ric} \geq (n-1) \rho, 
\end{equation} then $\bM$ is compact, with a finite
fundamental group, and diam$(\bM) \le \frac{\pi}{\sqrt{\rho}}$.
This is the celebrated Myer's theorem, which strengthens Bonnet's
theorem.

In what follows we state a sub-Riemannian counterpart of the 
Bonnet-Myer's compactness theorem, see \cite{BG1}. 

\begin{theorem}\label{T:BM}
Assume that the curvature-dimension inequality \eqref{cdi} be satisfied for $\rho_1> 0$, and that the  Hypothesis \ref{A:exhaustion}, \ref{A:main_assumption},  \ref{A:regularity} be valid. Then, the metric space $(\mathbb{M},d)$ is compact and we have \[ \emph{diam}\ \bM \le
2\sqrt{3} \pi \sqrt{
\frac{\rho_2+\kappa}{\rho_1\rho_2} \left(
1+\frac{3\kappa}{2\rho_2}\right)d } .
\]
\end{theorem}

 \section{Global volume doubling when Ricci $\ge 0$}\label{SS:volume}

Another fundamental tool in Riemannian geometry is the Bishop-Gromov volume comparison theorem. In what follows, given $\kappa\in \R$, we will indicate with $\M_\kappa$ the space  of constant sectional curvature $\kappa$, and with $V_\kappa(r)$ the volume of the geodesic ball $B_\kappa(r)$ in $\M_\kappa$. Given a Riemannian manifold with measure tensor $\mu$, for $x\in \M$ and $r>0$ we let
\[
V(x,r) = \mu(B(x,r)).
\]
\begin{theorem}[Bishop-Gromov comparison theorem]\label{T:BG}
Let $\M$ be a complete $n$-dimensional Riemannian manifold such that \emph{Ric}\ $\ge \rho$, $\rho \in \R$. Then, for every $x\in \M$ and every $r>0$ the function
\[
r \ \to\ \frac{V(x,r)}{V_{\frac{\rho}{n-1}}(r)}
\]
is non-increasing. 
\end{theorem}

\begin{corollary}\label{C:BG}
Let $\M$ be a complete $n$-dimensional Riemannian manifold with \emph{Ric}\ $\ge 0$. Then, for every $x\in \M$ and every $r>0$ the function
\[
r \ \to\ \frac{V(x,r)}{r^n}
\]
is non-increasing. As a consequence, one has
\begin{equation}\label{doubling}
V(x,2r) \le 2^n V(x,r),\ \ \ \ \ x\in \M,\ r>0,
\end{equation}
and since $\underset{r\to 0^+}{\lim}\ \frac{Vol(B(x,r))}{\omega_n r^n} = 1$, we also have the following \emph{maximum volume growth estimate}
\begin{equation}\label{growth}
V(x,r) \le \omega_n r^n,\ \ \ \ \ x\in \M,\ r>0.
\end{equation}
\end{corollary}

Theorem \ref{T:BG} and Corollary \ref{C:BG} play a pervasive role in the development of analysis on a Riemannian manifold with Ricci $\ge 0$. They are important, among other things, in the study of the spectrum of the Laplacian on a manifold, for establishing Gaussian bounds on the heat kernel, isoperimetric theorems, etc. 

In this section we intend to discuss a sub-Riemannian generalization of the doubling estimate \eqref{doubling}
 in Corollary \ref{C:BG} which has been established in \cite{BBG}, but see also \cite{BG2} for the Riemannian case. Remarkably, our approach shows that an inequality such as \eqref{doubling} above can be exclusively derived from the Bochner identity without a direct use of the theory of Jacobi fields. As a consequence, it provides a very flexible tool for situations in which the tools of Riemannian geometry are not readily available.

We illustrate the main essential point. From the semigroup property and the symmetry of the heat kernel we
have for any $y\in M$ and $t>0$
\[ p(y,y,2t) = \int_M  p(y,z,t)^2 d\mu(z).
\]

Consider now a function $h\in C^\infty_0(M)$ such that $0\le h\le
1$, $h\equiv 1$ on $B(x,\sqrt{t}/2)$ and $h\equiv 0$ outside
$B(x,\sqrt t)$. We thus have
\begin{align*}
P_t h(y) & = \int_M p(y,z,t) h(z) d\mu(z) 
\\
& \le \left(\int_{B(x,\sqrt t)}
p(y,z,t)^2 d\mu(z)\right)^{\frac{1}{2}} \left(\int_M h(z)^2
d\mu(z)\right)^{\frac{1}{2}}
\\
& \le p(y,y,2t)^{\frac{1}{2}} \mu(B(x,\sqrt t))^{\frac{1}{2}}.
\end{align*}

By taking $y=x$, and $t =r^2$ in the latter inequality, we obtain
\begin{equation}\label{ine7}
P_{r^2} \left(\mathbf 1_{B(x,r)}\right)(x)^2 \le P_{r^2} h(x)^2 \leq
p(x,x,2r^2)\ \mu(B(x,r)).
\end{equation}

Applying Corollary \ref{C:harnackheat}
 to $(y,t)\to p(x,y,t)$,
for every $y\in B(x,\sqrt t)$ we find
\[
p(x,x,t) \le  C p(x,y,2t).
\]
Integration in $y\in B(x,\sqrt t)$ gives
\[
p(x,x,t)\mu(B(x,\sqrt t)) \le C \int_{B(x,\sqrt
t)}p(x,y,2t)d\mu(y) \le C,
\]
where we have used $P_t1\le 1$. Letting $t = 4r^2$, we obtain from this the on-diagonal upper
bound \begin{equation}\label{odub} \mu(B(x,2r)) \le
\frac{C}{p(x,x,4r^2)}.
\end{equation}

At this point we combine \eqref{ine7} with \eqref{odub} to obtain
\begin{align}\label{final}
 \mu(B(x,2r)) & \le C\frac{p(x,x,2r^2)}{p(x,x,4r^2)} \frac{ \mu(B(x,r))}{P_{r^2} \left(\mathbf 1_{B(x,r)}\right)(x)^2}
 \\
 & \leq C^*  \frac{ \mu(B(x,r))}{P_{r^2} \left(\mathbf 1_{B(x,r)}\right)(x)^2},
 \notag
\end{align}
for every $x\in M$ and every $r>0$.

It is clear that we would obtain a sub-Riemannian counterpart of \eqref{doubling} if we could show that there exists $A\in (0,1)$, $K>0$, independent of $x\in M$ and $r>0$, such that
\[
P_{Ar^2} \left(\mathbf 1_{B(x,r)}\right)(x) \ge K.
\]
\textbf{Note:} The Harnack inequality in Theorem \ref{T:harnack} gives \begin{equation}\label{PP}
P_{r^2} \left(\mathbf 1_{B(x,r)}\right)(x) \ge C P_{Ar^2} \left(\mathbf 1_{B(x,r)}\right)(x).
\end{equation}

\begin{theorem}\label{T:estimee-P-boule}
Assume that the curvature-dimension inequality \eqref{cdi} be satisfied for $\rho_1\ge 0$ and that the Hypothesis \ref{A:exhaustion}, \ref{A:main_assumption},  \ref{A:regularity} be fulfilled. There exists a universal constant $0<A<1$ such that for every $x \in \mathbb{M}$, and $r>0$,
\[
P_{Ar^2} (\mathbf{1}_{B(x,r) })(x) \ge \frac{1}{2}.
\]
\end{theorem}

The proof of Theorem \ref{T:estimee-P-boule} is fairly complicated and it occupies large part of the work \cite{BBG}. For a much simpler account in the Riemannian setting we refer the reader to \cite{BG2}.
For future reference we record the following consequence of \eqref{ine7}, \eqref{PP}, and Theorem \ref{T:estimee-P-boule},
\begin{equation}\label{PPP}
p(x,x,2r^2) \ge \frac{C}{\mu(B(x,r))},\ \ \ \ \ x\in \M, r>0.
\end{equation}

With Theorem \ref{T:estimee-P-boule} in hands, following the arguments developed above, we obtain the following basic result.

\begin{theorem}[Global doubling property]\label{T:doubling}
Assume that the curvature-dimension inequality \eqref{cdi} be satisfied for $\rho_1\ge  0$, and that the  Hypothesis \ref{A:exhaustion}, \ref{A:main_assumption},  \ref{A:regularity} be valid. Then, the metric measure space $(\mathbb{M}, d , \mu)$ satisfies the global volume doubling property. More precisely, there exists a constant $C_1 = C_1(\rho_1,\rho_2,\kappa,d)>0$ such that for every $x \in \mathbb{M}$ and every $r>0$,
\[
\mu(B(x,2r)) \le C_1 \mu(B(x,r)).
\]
\end{theorem}


\section{Sharp Gaussian bounds, Poincar\'e inequality and parabolic Harnack inequality}\label{S:ogb}

The purpose of this section is to establish some optimal two-sided bounds for the heat kernel $p(x,y,t)$ associated with the subelliptic operator  $L$. Such estimates are reminiscent of those obtained by Li and Yau for complete Riemannian manifolds having Ric $\ge 0$. 
As a consequence of the  two-sided Gaussian bound for the heat kernel, we will derive a global Poincar\'e inequality and a localized parabolic Harnack inequality. 
Here is our main result.

\begin{theorem}\label{T:gb}
Suppose that the curvature-dimension inequality \eqref{cdi} be satisfied for $\rho_1\ge  0$, and that the  Hypothesis \ref{A:exhaustion}, \ref{A:main_assumption},  \ref{A:regularity} be valid. For any $0<\ve <1$
there exists a constant $C(\ve) = C(d,\kappa,\rho_2,\ve)>0$, which tends
to $\infty$ as $\ve \to 0^+$, such that for every $x,y\in \bM$
and $t>0$ one has
\[
\frac{C(\ve)^{-1}}{\mu(B(x,\sqrt
t))} \exp
\left(-\frac{D d(x,y)^2}{d(4-\ve)t}\right)\le p(x,y,t)\le \frac{C(\ve)}{\mu(B(x,\sqrt
t))} \exp
\left(-\frac{d(x,y)^2}{(4+\ve)t}\right).
\]
\end{theorem}

\begin{proof}
We begin by establishing the lower bound. First, from Corollary \ref{C:harnackheat}
 we obtain for all $y \in \bM$, $t>0$, and every $0<\ve <1$,
\beas 
p(x,y,t)&\geq& p(x,x,\ve t)  \ve^\frac{D}{2} \exp\left( -\frac{D}{d}\frac{d(x,y)^2}{(4-\ve)t}\right).
\eeas
We thus need to estimate $p(x,x,\ve t)$ from below. But this has already been done in \eqref{PPP}. Choosing $r>0$ such that $2r^2 = \ve t$, we obtain from that estimate
\[
p(x,x,\ve t) \ge \frac{C^*}{\mu(B(x,\sqrt{\ve/2} \sqrt t))},\ \ \ \ \ x\in \bM,\ t>0.
\]
On the other hand, since $\sqrt{\ve/2}<1$, by the trivial inequality $\mu(B(x,\sqrt{\ve/2} \sqrt t)) \le \mu(B(x,\sqrt t))$, we conclude
\[
p(x,y,t) \geq \frac{C^*}{ \mu(B(x,\sqrt t))}  \ve^\frac{D}{2} \exp\left( -\frac{D}{d}\frac{d(x,y)^2}{(4-\ve)t}\right).
\]
This proves the Gaussian lower bound.


For the Gaussian upper bound, we first observe recall that Theorem \ref{T:ub} gives for any $0<\ve'<1$
\begin{equation}\label{pabove}
p(x,y,t)\le \frac{C(d,\kappa,\rho_2,\ve')}{\mu(B(x,\sqrt
t))^{\frac{1}{2}} \mu(B(y,\sqrt
t))^{\frac{1}{2}}} \exp
\left(-\frac{d(x,y)^2}{(4+\ve')t}\right).
\end{equation}
At this point, by the triangle inequality and Theorem \ref{T:doubling} we find with $Q = \log_2 C_1$,
\beas
\mu(B(x,\sqrt{ t})) &\leq& \mu(B(y,d(x,y)+\sqrt{ t}))\\
                 &\leq& C_1 \mu(B(y,\sqrt{ t})) \left(\frac {d(x,y)+\sqrt{ t}}{\sqrt t} \right)^Q.
\eeas
This gives
$$
\frac{1}{\mu(B(y,\sqrt{ t}))}\leq \frac{C_1}{\mu(B(x,\sqrt{ t}))} \left(\frac {d(x,y)}{\sqrt{ t}}+1 \right)^Q.
$$
Combining this with \eqref{pabove} we obtain
\[
p(x,y,t)\le \frac{C_1^{1/2}C(d,\kappa,\rho_2,\ve')}{\mu(B(x,\sqrt
t))}  \left(\frac {d(x,y)}{\sqrt{ t}}+1 \right)^{\frac{Q}{2}} \exp
\left(-\frac{d(x,y)^2}{(4+\ve')t}\right).
\]
If now $0<\ve<1$, it is clear that we can choose $0<\ve'<\ve$ such that 
\[
\frac{C_1^{1/2}C(d,\kappa,\rho_2,\ve')}{\mu(B(x,\sqrt
t))}  \left(\frac {d(x,y)}{\sqrt{ t}}+1 \right)^{\frac{Q}{2}} \exp
\left(-\frac{d(x,y)^2}{(4+\ve')t}\right) \le  \frac{C^*(d,\kappa,\rho_2,\ve)}{\mu(B(x,\sqrt
t))} \exp
\left(-\frac{d(x,y)^2}{(4+\ve)t}\right),
\]
where $C^*(d,\kappa,\rho_2,\ve)$ is a constant which tends to $\infty$ as $\ve \to 0^+$. The desired conclusion follows by suitably adjusting the values of both $\ve'$ and of the constant in the right-hand side of the estimate.

\end{proof}

With Theorems \ref{T:doubling} and \ref{T:gb} in hands, we can now appeal to the results in \cite{FS}, \cite{KS3}, \cite{Gri},  \cite{SC}, \cite{St1}, \cite{St2}, \cite{St3}, see also the books \cite{GSC}, \cite{Gri2}. More precisely, from the developments in these papers it is by now well-known that in the context of strictly regular local Dirichlet spaces we have the equivalence between:

\

\begin{enumerate}
\item[(1)] A two sided Gaussian bounds for the heat kernel (like in Theorem \ref{T:gb});
\item[(2)] The  conjunction of the volume doubling property and the Poincar\'e inequality (see Theorem \ref{T:P} below);
\item[(3)] The parabolic Harnack inequality (see Theorem \ref{T:H} below).
\end{enumerate}

Thus,  thanks to Theorems  \ref{T:doubling} and \ref{T:gb},  we obtain the following form of Poincar\'e inequality.

\begin{theorem}\label{T:P}
Suppose that the curvature-dimension inequality \eqref{cdi} be satisfied for $\rho_1\ge  0$, and that the  Hypothesis \ref{A:exhaustion}, \ref{A:main_assumption},  \ref{A:regularity} be valid. Then, there exists a constant $C = C(d,\kappa,\rho_2)>0$ such that for every $x\in \M, r>0$, and $f\in C^\infty(\M)$ one has
\[
\int_{B(x,r)} |f(y) - f_r|^2 d\mu(y) \le C r^2 \int_{B(x,2r)} \Gamma(f)(y) d\mu(y),
\]
where we have let $f_r = \frac{1}{\mu(B(x,r))} \int_{B(x,r)} f d\mu$.
\end{theorem}

Since thanks to Theorem \ref{T:doubling} the space $(\M,\mu,d)$, where $d = d(x,y)$ indicates the sub-Riemannian distance \eqref{ds}, is a space of homogeneous type, and furthermore \eqref{ls} above guarantees that it is a length-space, then, arguing as in \cite{J}, from Theorem \ref{T:P} we obtain the following result.

\begin{corollary}\label{C:Pi}
Under the hypothesis of Theorem \ref{T:P} there exists a constant $C^* = C^*(d,\kappa,\rho_2)>0$ such that for every $x\in \M, r>0$, and $f\in C^\infty(\M)$ one has
\[
\int_{B(x,r)} |f(y) - f_r|^2 d\mu(y) \le C^* r^2 \int_{B(x,r)} \Gamma(f)(y) d\mu(y).
\]
\end{corollary}

Furthermore, the following scale invariant Harnack inequality for local solutions holds.

\begin{theorem}\label{T:H}
Assume that the curvature-dimension inequality \eqref{cdi} be satisfied for $\rho_1\ge  0$, and that the  Hypothesis \ref{A:exhaustion}, \ref{A:main_assumption},  \ref{A:regularity} be valid. If $u$ is a positive solution of the heat equation in a cylinder of the form $Q=(s,s+\alpha r^2) \times B(x,r)$ then
\begin{equation}\label{Ha}
\sup_{Q-} u\le C \inf_{Q+} u,
\end{equation}
where for some fixed $0 < \beta < \gamma <\delta<\alpha<\infty$ and $\eta \in (0,1)$,
\[
Q-=(s+\beta r^2,s+\gamma r^2)\times B(x,\eta r), Q+=(s+\delta r^2, s+\alpha r^2)\times B(x,\eta r).
\]
Here, the constant $C$ is independent of $x,r$ and $u$, but depends on the parameters $d, \kappa, \rho_2$, as well as on $\alpha, \beta, \gamma, \delta$ and $\eta$.
\end{theorem}


\section{Negatively curved manifolds}\label{S:nc}

In the previous sections we have exclusively discussed the case of sub-Riemannian manifolds with nonnegative Ricci curvature. In this section we present some of the main results in \cite{BBGM} relative to the case in which Ricci is bounded from below by a number which is allowed to be negative.

\begin{theorem}\label{T:main}
Suppose that the generalized curvature-dimension inequality \eqref{cdi} hold for some $\rho_1\in\R$, and that the  Hypothesis \ref{A:exhaustion}, \ref{A:main_assumption},  \ref{A:regularity} be valid. Then, there exist constants $C_1, C_2>0$, depending only on $\rho_1, \rho_2, \kappa, d$, for which one has for every $x,y\in \M$ and every $r>0$:
\begin{equation}\label{dcsr}
\mu(B(x,2r)) \le \ C_{1}\exp\left(C_{2}r^{2}\right)\mu(B(x,r)).
\end{equation}
The constant $C_2$ tends to zero as $\rho_1\to 0$, and thus \eqref{dcsr} contains in particular the estimate in Theorem \ref{T:doubling}.
\end{theorem}

In order to state the next result, we introduce a family  of control distances $d_\tau$ for $\tau \ge 0$. Given $x, y\in \M$, let us consider 
\[
S_\tau(x,y) =\{\gamma:[0,T]\to \M\mid \gamma\ \text{is subunit for}\ \Gamma+\tau^2 \Gamma^Z, \gamma(0) = x,\ \gamma(T) = y\}.
\]
A curve which is subunit for $\Gamma$ is obviously subunit for $ \Gamma+\tau^2 \Gamma^Z$, therefore thanks to the assumption \eqref{ls} above we have $S_\tau(x,y) \neq \varnothing$. We can then define
\begin{equation}
\label{distance-riem}
d_\tau(x,y) = \inf\{\ell_s(\gamma)\mid \gamma\in S_\tau(x,y)\}.
\end{equation} 
Note that $d(x,y)= d_0(x,y)$ and that, clearly: $d_\tau(x,y)\leq d(x,y)$.

\begin{theorem}\label{T:main2}
Suppose that the generalized curvature-dimension inequality hold for some $\rho_1\in\R$, and that the  Hypothesis \ref{A:exhaustion}, \ref{A:main_assumption},  \ref{A:regularity} be satisfied.  Let $\tau \ge 0$. Then, there  exists a constant $C(\tau)>0$, depending only on $\rho_1, \rho_2, \kappa, d$ and $\tau$, for which one has for every $x,y\in \M$:
\begin{equation}\label{NSWimp}
d\left(x,y\right)\leq C(\tau) \max\{ \sqrt{ d_{\tau}\left(x,y\right)},d_{\tau}\left(x,y\right)\}.
\end{equation}
\end{theorem}


\section{Geometric examples}\label{S:ge}

In this section we present several classes of sub-Riemannian spaces satisfying the generalized curvature-dimension inequality in Definition \ref{D:cdi} above. These examples constitute the central motivation of the present work.

\subsection{Riemannian manifolds}

As we have mentioned in the introduction, when $\bM$ is a $n$-dimensional complete Riemannian manifold with Riemannian distance $d_R$, Levi-Civita connection $\nabla$ and Laplace-Beltrami operator $\Delta$, our main assumptions hold trivially. It suffices in fact to choose $\Gamma^Z = 0$ to satisfy Hypothesis \ref{A:main_assumption} in a trivial fashion. Hypothesis \ref{A:exhaustion} is also satisfied since it is equivalent to assuming that $(\bM,d_R)$ be complete (observe in passing that the distance \eqref{di} coincides with $d_R$). Finally, with the choice $\kappa = 0$ and $\rho_1 = \rho$ the curvature-dimension inequality \eqref{cdi} reduces to \eqref{cdR}, which, as we have already observed, is implied by (and it is in fact equivalent to) the assumption Ric $\ \ge \rho$.

\subsection{The three-dimensional Sasakian models}\label{SS:3dimmod}

The purpose of this section is providing a first basic sub-Riemannian example which fits the framework of the present paper. This example was first studied in \cite{bakry-baudoin}. Given a number $\rho_1\in \R$, suppose that $\bG(\rho_1)$ be a three-dimensional Lie group whose Lie algebra $\mathfrak{g}$ has a
 basis $\left\{ X,Y ,Z \right\}$ satisfying:
\begin{itemize}
\item[(i)] $[X,Y]=Z$,
\item[(ii)] $[X,Z]= -\rho_1 Y$,
\item[(iii)] $[Y,Z]=\rho_1 X$.
\end{itemize} 
A sub-Laplacian on  $\bG(\rho_1)$ is  the left-invariant, second-order differential operator
\begin{equation}\label{slmodel}
L= X^{2}  + Y^{2}.
\end{equation}
In view of (i)-(iii) H\"ormander's theorem, see \cite{Ho}, implies that $L$ be hypoelliptic, although it fails to be elliptic at every point of $\bG(\rho_1)$. 
From \eqref{gamma} we find in the present situation 
 \[
 \Gamma(f) =\frac{1}{2}\big(L(f^2)-2fLf)= (Xf)^{2}+ (Yf)^{2}.
 \] 
If we define 
\[
\Gamma^Z(f,g)=Zf Zg,
\]
then from (i)-(iii) we easily verify that
\[
\Gamma(f,\Gamma^Z(f)) = \Gamma^Z(f,\Gamma(f)).
\]
We conclude that the Hypothesis \ref{A:main_assumption} is satisfied. It is not difficult to show that the Hypothesis \ref{A:exhaustion} is also fulfilled.

Using (i)-(iii) we  leave it to the reader to verify that
\begin{equation}\label{commLZ} 
[L,Z] = 0.
\end{equation}
By means of \eqref{commLZ} we easily find 
\begin{align*}
\Gamma_2^Z(f) &= \frac 12 L(\Gamma^Z(f)) - \Gamma^Z(f,Lf) = Zf [L,Z]f + (XZf)^2 +(YZf)^2
\\
& =  (XZf)^2 +(YZf)^2.
\end{align*}
Finally, from definition \eqref{gamma2} and from (i)-(iii) we obtain
\begin{align*}\label{Gamma2_3f}
\Gamma_2 (f) & = \frac 12 L(\Gamma(f)) - \Gamma(f,Lf)
\\
& = \rho_1 \Gamma(f) + (X^2 f)^2 + (YXf)^2 + (XYf)^2 +(Y^2 f)^2
\\
& + 2 Yf (XZf) - 2 Xf (YZ f).
\end{align*}
We now notice that 
\[
(X^2f)^2 + (YXf)^2 + (XYf)^2 +(Y^2 f)^2 = ||\nabla^2_H f||^2 + \frac 12 \Gamma^Z(f),
\]
where we have denoted by 
\[
\nabla^2_H f= \begin{pmatrix} X^2 f & \frac 12 (XY f + YXf)
\\
\frac 12 (XY f + YXf) & Y^2 f
\end{pmatrix}
\]
the symmetrized Hessian of $f$ with respect to the horizontal distribution generated by $X, Y$. 
Substituting this information in the above formula we find
\[
\Gamma_2 (f)  =  ||\nabla^2_H f||^2 + \rho_1 \Gamma(f) + \frac 12 \Gamma^Z(f) + 2 \big(Yf (XZf) - Xf (YZ f)\big).
\]
By the above expression for $\Gamma_2^Z(f)$, using Cauchy-Schwarz inequality, we obtain for every $\nu >0$
\[
|2 Yf (XZf) - 2 Xf (YZ f)| \le \nu \Gamma^Z_2(f) + \frac 1\nu \Gamma(f).
\]
Similarly, one easily recognizes that
\[
||\nabla^2_H f||^2 \ge \frac 12 (Lf)^2.
\]
Combining these inequalities,
we conclude that we have proved the following result. 
\begin{proposition}\label{P:cdmodels}
For every $\rho_1\in \R$ the Lie group $\mathbb{G}(\rho_1)$, with the sub-Laplacian $L$ in \eqref{slmodel}, satisfies the generalized  curvature dimension inequality \emph{CD}$(\rho_1,\frac{1}{2},1,2)$. Precisely, for every $f\in C^\infty(\bG(\rho_1))$ and any $\nu>0$ one has:
 \[
 \Gamma_{2}(f)+\nu \Gamma^Z_{2}(f) \ge \frac{1}{2} (Lf)^2 +\left(\rho_1 -\frac{1}{\nu}\right)  \Gamma (f)
 +\frac{1}{2} \Gamma^Z (f).
 \]
\end{proposition}

Proposition \ref{P:cdmodels} provides a basic motivation for Definition \ref{D:cdi}. It is also important to observe at this point that the Lie group $\bG(\rho_1)$ can be endowed with a natural CR structure. Denoting in fact with $\mathcal H$ the subbundle of $T\bG(\rho_1)$ generated by the vector fields $X$ and $Y$, the endomorphism $J$ of $\mathcal H$ defined by 
\[
J(Y) = X,\ \ \ \ J(X) = - Y,
\]
satisfies $J^2 = - I$, and thus defines a complex structure on $\bG(\rho_1)$. By choosing $\theta$ as the form such that \[
\text{Ker}\ \theta = \mathcal H,\   \  \ \text{and}\ \  \  d\theta(X,Y) = 1,
\]
we obtain a CR structure on $\bG(\rho_1)$ whose Reeb vector field is $-Z$. Thus, the above choice of $\Gamma^Z$ is canonical.

The pseudo-hermitian Tanaka-Webster torsion of $\bG(\rho_1)$ vanishes, and thus $(\bG(\rho_1),\theta)$ is a Sasakian manifold. It is also easy to verify that for the CR manifold $(\bG(\rho_1),\theta)$ the Tanaka-Webster horizontal sectional curvature is constant and equals $\rho_1$. 
The following three model spaces correspond respectively to the cases $\rho_1 = 1, \rho_1 = 0$ and $\rho_1 = -1$:  

\begin{itemize}
\item[1.] The Lie group $\mathbb{SU} (2)$ is the group of
$2 \times 2$, complex, unitary matrices of determinant $1$. 
\item[2.] The Heisenberg group $\mathbb{H}$ is the group of $3\times3$ matrices:
\[
\left(
\begin{array}
[c]{ccc}
~1~ & ~x~   & ~z ~\\
~0~ & ~1~   & ~y ~\\
~0~ & ~0~   & ~1 ~
\end{array}
\right)  ,\text{ \ }x,y,z\in\mathbb{R}.
\]
\item[3.] The Lie group $\mathbb{SL} (2)$ is the group of
$2 \times 2$, real matrices of determinant $1$. 
\end{itemize}

\subsection{Sub-Riemannian manifolds with transverse symmetries}\label{SS:subR}

We now turn our attention to a large class of sub-Riemannian manifolds, encompassing the three-dimensional model spaces discussed in the previous subsection. Theorem \ref{T:cd} below states that for these sub-Riemannian manifolds the generalized curvature-dimension inequality \eqref{cdi} does hold under some natural geometric assumptions which, in the Riemannian case, reduce to requiring a lower bound for the Ricci tensor. To achieve this result, some new Bochner type identities were established in \cite{BG1}. 

Let $\M$ be a smooth, connected  manifold equipped with a bracket generating distribution $\mathcal{H}$ of dimension $d$ and a fiberwise inner product $g$ on that distribution. The distribution $\mathcal{H}$ will be referred to as the set of \emph{horizontal directions}. 

We indicate with $\mathfrak{iso}$ the finite-dimensional Lie algebra of all sub-Riemannian Killing vector fields on $\bM$ (see \cite{Strichartz}). A vector field $Z\in \mathfrak{iso}$ if the one-parameter flow generated by it locally preserves the sub-Riemannian geometry defined by $(\mathcal{H},g )$. This amounts to saying that:
 \begin{itemize}
 \item[(1)] For every $x\in \bM$, and any $u,v \in \mathcal{H}(x)$, $\mathcal{L}_Z g (u,v)=0$;
 \item[(2)] If $X\in \mathcal H$, then $[Z, X]\in \mathcal H$.
 \end{itemize}
 In (1) we have denoted by $\mathcal L_Z g$ the Lie derivative of $g$ with respect to $Z$.  Our main geometric assumption is the following:
 
 \begin{assumption}
 There exists a Lie sub-algebra $\mathcal{V} \subset \mathfrak{iso}$, such that for every $x \in \bM$, 
 \[
 T_x \bM= \mathcal{H}(x) \oplus \mathcal{V}(x).
 \]
 \end{assumption}
 
 The distribution $\mathcal{V}$ will be referred  to as the set of \emph{vertical directions}. The dimension of $\mathcal{V}$ will be denoted by $\di$. 

The choice of an inner product on the Lie algebra $\mathcal{V}$ naturally endows $\bM$ with a Riemannian extension $g_R$ of $g$ that makes the decomposition $\mathcal{H}(x) \oplus \mathcal{V}(x)$ orthogonal. Although  $g_R$ is useful for computational purposes, the geometric objects that introduced in \cite{BG1}, like the sub-Laplacian $L$, the canonical connection $\nabla$ and the "Ricci" tensor $\mathcal R$, do not depend  on the choice of an inner product on $\mathcal{V}$. We refer to \cite{BG1} for a detailed geometric discussion. 

\begin{theorem}\label{T:cd}
Suppose that there exist constants $\rho_1
\in \mathbb{R}$, $\rho_2
>0$ and $\kappa \ge 0$ such that for every $f\in C^\infty(\bM)$:
\begin{equation}\label{riccibounds}
\begin{cases}
\mathcal{R}(f) \ge \rho_1 \Gamma (f) +\rho_2 \Gamma^Z (f),
\\
\mathcal{T}(f) \le \kappa \Gamma (f).
\end{cases}
\end{equation}
Then, the sub-Riemannian manifold $\bM$ satisfies the generalized curvature-dimension inequality \emph{CD}$(\rho_1,\rho_2,\kappa,d)$ in \eqref{cdi} with respect to the sub-Laplacian $L$ and the differential form $\Gamma^Z$.
\end{theorem}

In \cite{BG1} it was shown that, remarkably, the generalized curvature-dimension inequality \eqref{cdi} in Definition \ref{D:cdi} is equivalent to the geometric bounds \eqref{riccibounds} above. Here is the relevant result.

\begin{theorem}\label{P:tres_beau}
Suppose that there exist constants $\rho_1
\in \mathbb{R}$, $\rho_2
>0$ and $\kappa \ge 0$ such that $\M$ satisfy the generalized curvature-dimension inequality \emph{CD}$(\rho_1,\rho_2,\kappa,d)$. Then, $\M$ satisfies the geometric bounds \eqref{riccibounds}. As a consequence of this fact and of Theorem \ref{T:cd} we conclude that
\[
\emph{CD}(\rho_1,\rho_2,\kappa,d) \Longleftrightarrow \begin{cases}
\mathcal{R}(f) \ge \rho_1 \Gamma (f) +\rho_2 \Gamma^Z (f),
\\
\mathcal{T}(f) \le \kappa \Gamma (f).
\end{cases}
\]
\end{theorem}

\subsection{Carnot groups of step two}\label{SS:carnot}

Carnot groups of step 2 provide  a natural reservoir of sub-Riemannian manifolds with transverse symmetries. 
Let $\mathfrak{g}$ be a graded nilpotent Lie
algebra of step two.  This means that $\mathfrak{g}$ admits a splitting $\mathfrak g = V_1 \oplus V_2$, where $[V_1,V_1] = V_2$, and $[V_1,V_2]=\{0\}$. We endow $\mathfrak{g}$ with an inner product $\langle\cdot, \cdot \rangle$ with respect to which the decomposition $V_1 \oplus V_2$ is orthogonal.
  We denote by $e_1,...,e_d$  an orthonormal basis of $V_1$ and by $\varepsilon_1,...,\varepsilon_\di$ an orthonormal basis of $V_2$. Let $\bG$ be the connected and simply connected graded nilpotent Lie group associated with $\mathfrak{g}$. 
  Left-invariant vector fields in $V_2$ are seen to be transverse sub-Riemannian Killing vector fields of the horizontal distribution given by $V_1$.  The geometric assumptions  of the previous section are thus satisfied.

\begin{proposition}\label{P:carnotCD}
Let $\bG$ be a Carnot group of step two, with $d$ being the dimension of the horizontal layer of its Lie algebra. Then, $\bG$ satisfies the generalized curvature-dimension inequality \emph{CD}$(0,\rho_2, \kappa, d)$ (with respect to any sub-Laplacian $L$ on $\bG$),  
with $\rho_2>0$  and $\kappa\ge 0$ which solely depend on $\bG$. 
\end{proposition}
 
In particular, in our framework, every Carnot group of step two is a \textit{sub-Riemannian manifold with nonnegative Ricci tensor}.

\subsection{CR Sasakian manifolds}\label{SS:sasakian}

Another interesting class of sub-Riemannian manifolds with transverse symmetries is given by the class of CR Sasakian manifolds. For these manifolds one has the following result, established in \cite{BG1}.

\begin{theorem}\label{T:sasakian}
Let $\M$ be a Sasakian manifold, having real dimension $2n+1$. Assume that the Tanaka-Webster Ricci tensor is bounded from below by $\rho_1 \in \mathbb{R}$ on smooth functions, that is for every $f\in C^\infty(\bM)$ 
\[
\emph{Ric}(\nabla_\mathcal{H} f , \nabla_\mathcal{H} f ) \ge \rho_1 \|  \nabla_\mathcal{H} f \|^2.
\]
Then, $\bM$ satisfies the generalized curvature-dimension inequality  \emph{CD}$(\rho_1,\frac{n}{2},1,2n)$.
\end{theorem}

\end{document}